\documentclass[12pt,a4paper]{amsart}

\newtheorem{theo+}              {Theorem}           [section]
\newtheorem{prop+}  [theo+]     {Proposition}
\newtheorem{coro+}  [theo+]     {Corollary}
\newtheorem{lemm+}  [theo+]     {Lemma}
\newtheorem{exam+}  [theo+]     {Example}
\newtheorem{rema+}  [theo+]     {Remark}
\newtheorem{defi+}  [theo+]     {Definition}
\newtheorem{clai+}  [theo+]     {Claim}
\newenvironment{theorem}{\begin{theo+}}{\end{theo+}}
\newenvironment{proposition}{\begin{prop+}}{\end{prop+}}
\newenvironment{corollary}{\begin{coro+}}{\end{coro+}}
\newenvironment{lemma}{\begin{lemm+}}{\end{lemm+}}

\usepackage{amsthm}
\theoremstyle{plain} \theoremstyle{remark}
\newtheorem{remark}{Remark}
\newtheorem{example}{Example}

\def \r{\mbox{${\mathbb R}$}}
\def\E{/\kern-1.0em \equiv }

\author{Ze-Ping Wang$^{*}$ and Ye-Lin Ou$^{**}$ }

\address{School of Mathematical Sciences,\newline\indent Guizhou
Normal University,\newline\indent Guiyang 550025,\newline\indent
People's Republic of China
\newline\indent E-mail:zpwzpw2012@126.com \;(Wang)\\
\newline
\newline\indent Department of
Mathematics,\newline\indent Texas A $\&$ M University-Commerce,
\newline\indent Commerce TX 75429,\newline\indent USA.\newline\indent
E-mail:yelin.ou@tamuc.edu \;(Ou)}
\thanks{*Supported by the Natural Science Foundation of China (No. 11861022). \\
\indent** Supported by a grant from the Simons Foundation ( 427231,
Ye-Lin Ou).}
\date{11/12/2023}
\begin{document}
\title[Biharmonic Riemannian submersions from a BCV space]{Biharmonic Riemannian submersions from a $3$-dimensional BCV space}


\maketitle

\section*{Abstract}
\begin{quote}
{\footnotesize BCV spaces are a family of $3$-dimensional  Riemannian manifolds which include six of Thurston's eight geometries. In this paper, we give a complete classification of proper biharmonic Riemannian submersions from a $3$-dimensional BCV space by proving that such biharmonic maps exist only  in the cases of $H^2\times\r\to \r^2$,\;or  $\widetilde{SL}(2,\r)\to \r^2$. In each of these two cases, we are able to construct a family of infinitely many proper biharmonic Riemannian submersions. Our results on one hand, extend a previous result of the authors which gave a complete classification of proper biharmonic Riemannian submersions from a $3$-dimensional space form, and on the other hand, can be viewed as the dual study of biharmonic surfaces (i.e., biharmonic isometric immersions) in a BCV space studied in some recent literature.}

\end{quote}

\maketitle

\section{Introduction}
All manifolds, maps, and  tensor fields studied in this paper are assumed
to be smooth unless there is an otherwise statement.\\

A harmonic map $\varphi:(M, g)\to (N,
h)$ from a compact Riemannian manifold
$(M, g)$ into another Riemannian manifold $(N, h)$ is a critical point
of the energy functional defined by
\begin{equation}\nonumber
E\left(\varphi,\Omega \right)= \frac{1}{2} {\int}_{\Omega}
\left|{\rm d}\varphi \right|^{2}{\rm d}x.
\end{equation}
So,  it is a solution of the corresponding Euler-Lagrange equation. This
equation is given by the vanishing of the tension field (see, e.g., \cite{BW1})
\begin{equation}\notag
\tau(\varphi)={\rm
Trace}_{g}\nabla {\rm d} \varphi=0.
\end{equation}

A map $\varphi:(M, g)\to (N,
h)$\; between Riemannian manifolds is called {\em biharmonic map} if $\varphi|_{\Omega}$ is  a critical point of the
bienergy
\begin{equation}\nonumber
E^{2}\left(\varphi,\Omega \right)= \frac{1}{2} {\int}_{\Omega}
\left|\tau(\varphi) \right|^{2}{\rm d}x
\end{equation}
for every compact subset $\Omega$ of $M$.  By calculating the first
variation of the functional (see \cite{Ji}),  one concludes that $\varphi$
is  biharmonic  if and only if its bitension field vanishes
identically, i.e.,
\begin{equation}\label{BT1}
\tau^{2}(\varphi):={\rm
Trace}_{g}(\nabla^{\varphi}\nabla^{\varphi}-\nabla^{\varphi}_{\nabla^{M}})\tau(\varphi)
- {\rm Trace}_{g} R^{N}({\rm d}\varphi, \tau(\varphi)){\rm d}\varphi
=0,
\end{equation}
where $R^{N}$ denotes the curvature operator of $(N, h)$ defined by
$$R^{N}(X,Y)Z=
[\nabla^{N}_{X},\nabla^{N}_{Y}]Z-\nabla^{N}_{[X,Y]}Z.$$\\
Clearly, a harmonic map is always a biharmonic map, so we call those biharmonic maps which are not harmonic proper biharmonic maps.\\

For more account on  basic examples and properties of biharmonic maps, some recent progress on biharmonic submanifolds (i.e., biharmonic isometric immersions), biharmonic conformal maps, biharmonic maps into spheres, biharmonic maps with symmetry, Liouville type and unique continuation theorems for biharmonic maps, see the recent book \cite{Ou2} and the vast references therein.\\

 Biharmonic Riemannian submersions were first studied by Oniciuc in \cite{CO}. By using the integrability data of a special orthonormal frame adapted to a Riemannian submersion, we proved in  \cite{WO} that  a Riemannian submersion from a $3$-dimensional space form into a surface is biharmonic if and only if it is harmonic.  In a recent paper \cite{MO},  Akyol and Ou  studied biharmonicity of a general Riemannian submersion and obtained biharmonic equations for Riemannian submersions with one-dimensional fibers and Riemannian submersions with
basic mean curvature vector fields of fibers. In particular, in \cite{MO}, biharmonic Riemannian
submersions from $(n+1)$-dimensional spaces with one-dimensional fibers were studied using
the  integrability data.  In a recent paper \cite{EO}, the authors found many examples of biharmonic Riemannian submersions in their study of generalized harmonic morphisms which are maps between Riemannian manifolds that pull back local harmonic functions to local biharmonic functions.\\

A 3-dimensional Bianchi-Cartan-Vranceeanu  (3-dimensional BCV space for short) (see e.g., \cite{BDI}, \cite{CMOP}) refers to the 3-dimensional Riemannian manifold
$$M^3_{m,\;l}=(\r^{3},g=\frac{dx^2+dy^2}{[1+m(x^2+y^2)]^2}+[dz+\frac{l}{2}\frac{y
dx-x dy}{1+m(x^2+y^2)}]^2),$$
 which include the well-known model spaces $\r^3, S^3, S^2\times\r, H^2\times\r, \widetilde{SL}(2,\r), Nil,$\\ $ SU(2)$.  BCV spaces are important because they include, on one hand,  all 3-dimensional homogeneous metrics with group of isometries of dimension 4 or 6, except for those of constant negative sectional curvature, and on the other hand, six of  Thurston's  the eight 3-dimensional geometries.\\

In this paper, we give a complete classification of proper biharmonic Riemannian submersions from a $3$-dimensional BCV space  by proving that such biharmonic maps exist only  in the cases of $H^2\times\r\to \r^2$,\;or  $\widetilde{SL}(2,\r)\to \r^2$. In each of these two cases, we are able to construct a family of infinitely many proper biharmonic Riemannian submersions. Our results on one hand, extend the results in \cite{WO} where a complete classification of proper biharmonic Riemannian submersions from a $3$-dimensional space form was obtained, and on the other  hand, can be viewed as the dual study of biharmonic surfaces (i.e., biharmonic isometric immersions) in a BCV space studied in \cite{CMO1,CMO2,FMO,Ou4}.

\section{ Preliminaries}

Let $\pi:( M^3 , g)\to(N^2,h)$ be a Riemannian
submersion. Recall that  a local orthonormal frame is an adapted frame to
the Riemannian submersion $\pi$ if the vector fields in the frame
that are tangent to the horizontal distribution are basic, that is,
they are $\pi$-related to a local orthonormal frame in the base
space. Such a frame always exists (cf. e.g., \cite{BW1}). Suppose
 that $\{\alpha_1,\; \alpha_2, \;e_3\}$ is an adapted orthonormal frame
with $e_3$ being vertical. Then, it is a known fact (see \cite{On1})
that $[\alpha_1,e_3]$ and $ [\alpha_2,e_3]$ are vertical and $[\alpha_1,\alpha_2]$ is
$\pi$-related to $[\eta_1, \eta_2]$, where
$\{\eta_1, \eta_2\}$ is an orthonormal frame in the
base space. Suppose that $ [\eta_1,\eta_2]=H_1\eta_1+H_2\eta_2,$
for $H_1, H_2\in C^{\infty}(N)$ and use the notations $h_i=H_i\circ
\pi, i=1, 2$. Then, we obtain
\begin{equation}\label{R1}
\begin{array}{lll}
[\alpha_1,e_3]=\tau_1e_3,\;
[\alpha_2,e_3]=\tau_2e_3,\;
[\alpha_1,\alpha_2]=h_1 \alpha_1+h_2\alpha_2-2\sigma e_3.
\end{array}
\end{equation}
where  $ h_1, h_2, \tau_1,\;\tau_2\;{\rm and}\; \sigma \in C^{\infty}(M)$ are the integrability data of the adapted frame of the Riemannian submersion.\\

Using an adapted frame and the associated integrability data the biharmonicity of a Riemannian submersion  can be described as follows.

\begin{lemma}(see \cite{WO})\label{Lem1}
Let $\pi:(M^3,g)\to(N^2,h)$ be a Riemannian submersion
with the adapted frame $\{e_1,\; e_2, \;e_3\}$ and the integrability
data $ \{f_1,\; f_2,\; \kappa_1,\;\kappa_2\;\sigma\}$. Then, the
Riemannian submersion $\pi$ is biharmonic if and only if
\begin{equation}\label{lem1}
\begin{cases}
-\Delta^{M}\kappa_1-2\sum\limits_{i=1}^{2}f_i e_i(\kappa_2)-\kappa_2\sum\limits_{i=1}^{2}\left(e_i( f_i)
-\kappa_i f_i\right)+\kappa_1\left(-K^{N}+\sum\limits_{i=1}^{2}f_{i}^{2}\right)
=0,\\
-\Delta^{M}\kappa_2+2\sum\limits_{i=1}^{2}f_i e_i(\kappa_1)+\kappa_1\sum\limits_{i=1}^{2}(e_i( f_i)
-\kappa_i f_i)+\kappa_2\left(-K^{N}+\sum\limits_{i=1}^{2}f_{i}^{2}\right)=0,
\end{cases}
\end{equation}
where
$K^{N}=R^{N}_{1212}\circ\pi=e_1(f_2)-e_2(f_1)-f_{1}^{2}-f_{2}^{2}$ is  the Gauss curvature of Riemannian manifold $(N^2,h)$.
\end{lemma}

One can easily check that on a BCV space $M^3_{m,\;l}=(\r^{3},g=\frac{dx^2+dy^2}{[1+m(x^2+y^2)]^2}+[dz+\frac{l}{2}\frac{y
dx-x dy}{1+m(x^2+y^2)}]^2)$ we have a globally defined  orthonormal frame
\begin{equation}\label{X190}\notag
E_{1}=F\frac{\partial}{\partial
x}-\frac{ly}{2}\frac{\partial}{\partial z},\;E_{2}=F
\frac{\partial}{\partial y}+\frac{lx}{2}\frac{\partial}{\partial
z},\;E_{3}=\frac{\partial}{\partial z},
\end{equation}
where $F=1+m(x^2+y^2)$.  \\
A direct computation gives the Lie brackets
\begin{equation}\label{Lie}\notag
[E_1,E_2]=-2myE_{1}+2mxE_{2}+lE_{3},\;\; all\;\;
other\;\;[E_i,E_j]=0,\;i,j=1,2,3,
\end{equation}
and  the Levi-Civita connection of  BCV 3-spaces as
\begin{equation}\label{BCV1}
\begin{cases}
\nabla_{E_{1}}E_{1}=2myE_{2},\;\;\nabla_{E_{2}}E_{2}=2mxE_{1},\;\;\\
\nabla_{E_{1}}E_{2}=-2myE_{1}+\frac{l}{2}E_{3},\;\;
\nabla_{E_{2}}E_{1}=-2mxE_{2}-\frac{l}{2}E_{3},\;\;\\
\nabla_{E_{3}}E_{1}=\nabla_{E_{1}}E_{3}=-\frac{l}{2}E_{2},\;\;\nabla_{E_{3}}E_{2}=\nabla_{E_{2}}E_{3}=\frac{l}{2}E_{1},\;\;\\
\;\; all \;\; other\;\;
\nabla_{E_i}E_j=0,\;i,j=1,2,3.\;\;\\
\end{cases}
\end{equation}
One can further compute the components of the
Riemannian curvature to obtain
\begin{equation}\label{BCV2}
\begin{array}{lll}
 R_{1212}=g(R(E_{1},E_{2})E_{2},E_{1})=4m-\frac{3l^2}{4},\\
R_{1313}=g(R(E_{1},E_{3})E_{3},E_{1})=R_{2323}=g(R(E_{2},E_{3})E_{3},E_{2})=\frac{l^2}{4},\\
\;\; all \;\; other\;\;R_{ijkl}=g(R(E_{k},E_{l})E_{j},E_{i})=0,\;i,j,k,l=1,2,3.
\end{array}
\end{equation}

The key ingredient in the proofs of our main theorems is the proof of the  existence of certain preferred orthonormal frame $\{e_1, e_2, e_3\}$ which has ``nice" relationship to the natural orthonormal frame $\{E_1, E_2, E_3\}$ on a BCV space.

Note that there are many orthonormal frames $\{e_1,
e_2, \;e_3\}$ with $e_3$ vertical to a Riemannian
submersion $\pi:(M^3,g) \to (N^2,h)$ which are not adapted frames since $e_1, e_2$ may not be basic. We call such an orthonormal frame a natural orthonormal frame. We will first compute the ( generalized) integrability data of a natural orthonormal frame using their relationship with an adapted frame.

\begin{proposition}\label{Pro2}
Let $\pi:(M^3,g) \to (N^2,h)$ be a Riemannian
submersion and  $\{\alpha_1, \alpha_2, e_3\}$ be an adapted  frame with vertical $e_3$  and the  integrability data $\{h_1, h_2,\tau_1,\tau_2, \sigma\} $.
Then, for any natural orthonormal frame $\{e_1=a\alpha_1+b\alpha_2,
e_2=-b\alpha_1+a\alpha_2, \;e_3\}$ on $ M^3$, we have
\begin{equation}\label{PR0}
\begin{array}{lll}
[e_1,e_3]=f_{3}e_2+\kappa_1e_3,\;
[e_2,e_3]=-f_{3}e_1+\kappa_2e_3,\;
[e_1,e_2]=f_1 e_1+f_2e_2-2\sigma e_3,
\end{array}
\end{equation}
and
\begin{equation}\label{PR1}
\begin{cases}
f_1=-ae_1(b)+be_1(a)+bh_2+ah_1,\\
 f_2=-ae_2(b)+be_2(a)+ah_2-bh_1,\\
f_{3}=be_3(a)-ae_3(b),\\
 \kappa_1=a\tau_1+b\tau_2,\;
 \kappa_2=-b\tau_1+a\tau_2,
\end{cases}
\end{equation}
where $a, b, f_{3},\;f_1, f_2, \kappa_1,\;\kappa_2,\;{\rm and}\;\sigma
\in C^{\infty}(M)$. Furthermore, the  orthonormal frame $\{e_1,\; e_2, \;e_3\}$ is an adapted frame  if and only if $f_{3}=0$.
\end{proposition}
\begin{proof}
 Since  $\{\alpha_1, \alpha_2\}$  and $\{e_1,\; e_2\}$ are two orthonormal bases of the horizontal distribution, we have
\begin{equation}\label{JR2}
e_1=a\alpha_1+b\alpha_2,\;e_2=-b\alpha_1+a\alpha_2,
\end{equation}
where $a, b\in C^\infty(M)$, and $a^2+b^2=1$.\\
It follows that $a\alpha_i (a)+b\alpha_i (b)=\frac{1}{2}\alpha_i(a^2+b^2)=0,\; i=1,2$, and $a e_3 (a)+be_3 (b)=0$.\\
Using Eqs. (\ref{R1}) and  (\ref{JR2}), together with $a\alpha_i (a)+b\alpha_i (b)=0$,  a straightforward computation yields
\begin{equation}\label{JR3}\notag
\begin{array}{lll}
[e_1,e_2]=[a\alpha_1+b\alpha_2,-b\alpha_1+a\alpha_2] \\
=\left(-a\alpha_1(b)+b\alpha_1(a)\right)\alpha_1
+\left(b\alpha_2(a)-a\alpha_2(b)\right)\alpha_2+[\alpha_1,\alpha_2]\\
=\left(-a\alpha_1(b)+b\alpha_1(a)+h_1\right)\alpha_1
+\left(b\alpha_2(a)-a\alpha_2(b)+h_2\right)\alpha_2-2\sigma e_3.
\end{array}
\end{equation}

A further computation gives
\begin{equation}\label{JR4}\notag
\begin{array}{lll}
\langle[e_1,e_2],e_1\rangle=\langle[e_1,e_2],a\alpha_1+b\alpha_2\rangle
=-ae_1(b)+be_1(a)+bh_2+ah_1,\\
\langle[e_1,e_2],e_2\rangle=\langle[e_1,e_2],-b\alpha_1+a\alpha_2\rangle
= -ae_2(b)+be_2(a)+ah_2-bh_1,\\
\langle[e_1,e_2],e_3\rangle=-2\sigma,
\end{array}
\end{equation}
which gives
\begin{equation}\label{JR7}\notag
\begin{array}{lll}
[e_1,e_2]=f_1e_1+f_2e_2-2\sigma e_3,
\end{array}
\end{equation}
where $f_1=-ae_1(b)+be_1(a)+bh_2+ah_1,\;f_2=-ae_2(b)+be_2(a)+ah_2-bh_1$.\\
A similar calculation using Eqs. (\ref{R1}), (\ref{JR2}) gives
\begin{equation}\label{JR8}
\begin{array}{lll}
[e_1,e_3]=[a\alpha_1+b\alpha_2,e_3]
=-e_3(a)\alpha_1-e_3(b)\alpha_2+(a\tau_1+b\tau_2)e_3,\\

[e_2,e_3]=[-b\alpha_1+a\alpha_2,e_3]
=e_3(b)\alpha_1-e_3(a)\alpha_2+(-b\tau_1+a\tau_2)e_3.
\end{array}
\end{equation}
A further computation using (\ref{JR8}), together with $a e_3 (a)+be_3 (b)=0$,  yields
\begin{equation}\label{JR10}\notag
\begin{array}{lll}
\langle[e_1,e_3],e_1\rangle=\langle[e_1,e_3],a\alpha_1+b\alpha_2\rangle=-(ae_3(a)+be_3(b))=0,\\

\langle[e_1,e_3],e_2\rangle=\langle[e_1,e_3],-b\alpha_1+a\alpha_2\rangle=be_3(a)-ae_3(b),\\

\langle[e_1,e_3],e_3\rangle=a\tau_1+b\tau_2,\;
\langle[e_2,e_3],e_3\rangle=-b\tau_1+a\tau_2,\\

\langle[e_2,e_3],e_1\rangle=\langle[e_2,e_3],a\alpha_1+b\alpha_2\rangle=-(-ae_3(b)+be_3(a)),\\

\langle[e_2,e_3],e_2\rangle=\langle[e_2,e_3],-b\alpha_1+a\alpha_2\rangle=-(be_3(b)+ae_3(a))=0,

\end{array}
\end{equation}
from which, we obtain
\begin{equation}\label{JR16}\notag
[e_1,e_3]=f_{3}e_2+\kappa_1e_3,\;[e_2,e_3]=-f_{3}e_1+\kappa_2e_3,
\end{equation}
where $f_{3}=be_3(a)-ae_3(b), \kappa_1=a\tau_1+b\tau_2, \kappa_2=-b\tau_1+a\tau_2$.
Thus,  we get (\ref{PR0}).\\
For the last statement, we note that if $ab=0$, then $\{e_1, e_2\}$ is obtained from $\{\alpha_1, \alpha_2\}$ by a rotation of $\pi/2$ or $\pi$ and hence they are basic. Otherwise, it follows from $f_{3}=be_3(a)-ae_3(b)$ that $f_{3}=0$ is equivalent to $e_3\left(\frac{a}{b}\right)=0$. This means the angle between $e_i$ and $\alpha_1$ is constant along the fibers, so Lemma 1.2 in \cite{Es} applies to conclude that  both $e_1,e_2$  are basic and hence the frame $\{e_1,\; e_2, e_3\}$ is an adapted frame. Thus, we complete the proof of the proposition.
\end{proof}

\begin{remark}\label{rem0}
Let $\pi:(M^3,g) \to (N^2,h)$ be a Riemannian submersion with an orthonormal frame $\{e_1,\;e_2, \;e_3\}$ and $e_3$ being  vertical. A straightforward computation using (\ref{PR0}) and Koszul formula gives
\begin{equation}\label{R2}
\begin{cases}
\nabla_{e_{1}} e_{1}=-f_1e_2,\;\;\nabla_{e_{1}} e_{2}=f_1
e_1-\sigma e_{3},\;\;\nabla_{e_{1}} e_{3}=\sigma
e_{2},\\  \nabla_{e_{2}} e_{1}=-f_2 e_{2}+\sigma
e_3,\;\;\nabla_{e_{2}} e_{2}=f_2 e_{1}, \;\;\nabla_{e_{2}}
e_{3}=-\sigma e_{1},\\
\nabla_{e_{3}}e_{1}=-\kappa_1e_{3}+(\sigma-f_{3}) e_{2}, \nabla_{e_{3}} e_{2}= -(\sigma-f_{3})
e_{1}-\kappa_2 e_3, \nabla_{e_{3}} e_{3}=\kappa_1 e_{1}+\kappa_2e_2.
\end{cases}
\end{equation}
One can  check  that using (\ref{R2}) and applying the Jacobi identities to the frame $\{e_1, e_2, e_3\}$ yield
 \begin{equation}\label{Jac}
 \begin{array}{lll}
e_3(f_1)+(\kappa_1+f_2)f_{3}-e_1(f_{3})=0,\;
e_3(f_2)+(\kappa_2-f_1)f_{3}-e_2(f_{3})=0\\
2 e_3(\sigma)+\kappa_1f_1+\kappa_2f_2+e_2(\kappa_1)-e_1(\kappa_2)=0,
\end{array}
\end{equation}
and the terms of the curvature tensor
\begin{equation}\label{RC}
\begin{cases}
R^{M}(e_1,e_3,e_1,e_2)=-e_1(\sigma)+2\kappa_1\sigma,\\
R^{M}(e_1,e_3,e_1,e_3)=e_1(\kappa_1)+\sigma^2-\kappa_{1}^2+\kappa_2f_1,\;\\
R^{M}(e_1,e_3,e_2,e_3)=e_1(\kappa_2)-e_3(\sigma)-\kappa_{1}f_{1}-\kappa_1\kappa_2,\;\\
R^{M}(e_1,e_2,e_1,e_2)=e_1(f_2)-e_2(f_1)-f_{1}^{2}-f_{2}^{2}+2f_{3}\sigma-3\sigma^2,\\
R^{M}(e_1,e_2,e_2,e_3)=-e_2(\sigma)+2\kappa_2\sigma,\\
R^{M}(e_2,e_3,e_1,e_3)=e_2(\kappa_{1})+e_3(\sigma)+\kappa_2 f_2-\kappa_1 \kappa_2,\\
R^{M}(e_2,e_3,e_2,e_3)=\sigma^{2}+e_2(\kappa_2)-\kappa_1f_2- \kappa_2^2.
\end{cases}
\end{equation}
By the 4th equation of (\ref{RC}) and O'Neill curvature formula for a Riemannian submersion \cite{On1} we have the Gauss curvature of the base space  as

\begin{equation}\label{GCB}
K^N=e_1(f_2)-e_2(f_1)-f_1^2-f_2^2+2f_{3}\sigma.
\end{equation}
Since the Gauss curvature of the space $K^N=R^N_{1212}\circ\pi$, then  it is constant along the fibers, that is
\begin{equation}\label{GCB0}
e_3(K^N)=e_3[e_1(f_2)-e_2(f_1)-f_1^2-f_2^2+2f_{3}\sigma]=0.
\end{equation}

If $f_{3}=0$, i.e., the orthonormal  frame $\{e_1,\; e_2, \;e_3\}$ is an adapted frame, then the
Gauss curvature of the base space turns into

\begin{equation}\label{GCB1}
K^N=e_1(f_2)-e_2(f_1)-f_1^2-f_2^2.
\end{equation}
\end{remark}

To further understand the properties of the orthonormal frame $\{e_1,
e_2, \;e_3\}$ with $e_3$ vertical to a Riemannian
submersion $\pi:(M^3,g) \to (N^2,h)$, we need the following results.

\begin{proposition}\label{Pro3}
(see, e.g.,  \cite{A,On2}) Let $\pi:(M,g)\longrightarrow (N,h)$ be a
Riemannian submersion. Let $\gamma:I\longrightarrow N$ be a geodesic
in $N$ where $I$ is an open interval around 0 in $\r$. Let
$\gamma(0)=b$ and $ p\in \pi^{-1}(b)$ be an arbitrary point in the
fibre above $b$. If $\widetilde{\gamma}$ is the horizontal lift of
$\gamma$ passing through $p$, then $\widetilde{\gamma}$ is also a
geodesic in $M$.

\end{proposition}
\begin{corollary}\label{Coro3}  (see, e.g., \cite{A,On2})
 Let $\pi:(M,g)\to (N,h)$ be a Riemannian
submersion. If $\gamma$ is a geodesic in $M$ which is horizontal at
a point, then $\gamma$ is horizontal throughout and $\pi\circ
\gamma$ is a geodesic in $N$.
\end{corollary}

\begin{proposition}\label{Pro4}(see, e.g., \cite{Do,MH})
Let $\gamma:I\to(N^2,h=du^2+e^{2\lambda(u,v)}dv^2)$ be a geodesic curve with $\bar{\theta}$ being the angle between the geodesic curve $\gamma$ and  $u$-curves. Then, we have
\begin{equation}\notag
\frac{d\bar{\theta}}{dv}=-\lambda_u e^{\lambda}.
\end{equation}
Furthermore, if denote by $\eta_1=\frac{\partial}{\partial u}, \;\eta_2=e^{-\lambda}\frac{\partial}{\partial v}$ an orthonormal
frame on $(N^2,h)$, then
\begin{equation}\label{Z0}\notag
\eta_2(\bar{\theta})=-\lambda_u.
\end{equation}
\end{proposition}

 Now we are ready to prove  the following technical lemma which plays a crucial role in the proofs of our main theorems.

\begin{lemma}\label{L1}
Let $\pi:(M^3,g)\to (N^2,h)$ be a Riemannian submersion
from 3-manifolds  with  an orthonormal frame $\{e_1, e_2, e_3\}$ and
$ e_3$ being vertical. If $\nabla_{e_{1}}e_{1}=0$, then either
$(i)$ $\nabla_{e_{2}}e_{2}=0$,\;or
$(ii)$ $\nabla_{e_{2}}e_{2}\not\equiv0$, and the orthonormal
frame $\{e_1, e_2, e_3\}$  is an adapted frame to the Riemannian submersion $\pi$.
\end{lemma}

\begin{proof}
 It is well known that a 2-manifold $(N^2,h)$ can be locally expressed as $(N^2,h)=(\r^2,du^2+e^{\lambda(u,v)}dv^2)$ with respect to local coordinates $(u,v)$. Then, we have an orthonormal frame
\begin{equation}\label{Z1}\notag
\eta_1=\frac{\partial}{\partial u},\; \eta_2=e^{-\lambda}\frac{\partial}{\partial v}
\end{equation}
with
\begin{equation}\label{Z2}\notag
[\eta_1,\eta_2]=H_1\eta_2+H_2\eta_2=H_2\eta_2,
\end{equation}
where $H_1=0,\;H_2=-\lambda_u$.\\
Let  $\alpha_1, \alpha_2$ be the horizontal lifts of  $\eta_1, \eta_2$ respectively. Then, we have an adapted orthonormal frame $\{\alpha_1, \alpha_2, e_3\}$ with vertical $e_3$,  and $\alpha_1, \alpha_2$  $\pi$-related to $\eta_1, \eta_2$ respectively, and
\begin{equation}\label{Z3}\notag
\begin{array}{lll}
[\alpha_1,e_3]=\tau_1e_3,\;[\alpha_2,e_3]=\tau_2e_3,\;[\alpha_1,\alpha_2]=h_2\alpha_2-2\sigma e_3,\\
 h_1=H_1\circ\pi=0,\; h_2=H_2\circ\pi,\;
 e_3(h_1)=e_3(h_2)=0,
\end{array}
\end{equation}
where $\tau_1, \tau_2,\;h_2\;{\rm and} \;\sigma$  are the integrability
data.\\
Since both $\{e_1, e_2, e_3\}$ and $\{\alpha_1, \alpha_2, e_3\}$ are  orthonormal frames on $M^3$, we can assume
\begin{equation}\label{Z4}
\begin{array}{lll}
e_1=a\alpha_1+b\alpha_2,\;e_2=-b\alpha_1+a\alpha_2,
\end{array}
\end{equation}
where $a^2+b^2=1$.\\
Using Proposition \ref{Pro2} with $h_1=0$, the assumption that $\nabla_{e_1}e_1=0$, and (\ref{R2}) we have
\begin{equation}\label{Z5}
\begin{array}{lll}
[e_1,e_3]=f_{3}e_2+\kappa_1e_3,\; [e_2,e_3]=-f_{3}e_1+\kappa_2e_3,\;
[e_1,e_2]=f_2e_2-2\sigma e_3,\\
f_1=-ae_1(b)+be_1(a)+bh_2=0,\;
 f_2=-ae_2(b)+be_2(a)+ah_2.
\end{array}
\end{equation}

A further computation using $ae_1(a)+be_1(b)=0,\;ae_2(a)+be_2(b)=0$ and $a^2+b^2=1$, and the last two equations of (\ref{Z5}) gives $
e_1(a)=-b^2h_2$, and $e_2(a)=bf_2-abh_2.$
It follows from this and  (\ref{Z4}) that
\begin{equation}\label{Z7}
\alpha_2(a)=be_1(a)+ae_2(a)=abf_2-bh_2.
\end{equation}

Now, we  show that $abf_2=0$.\\
{\bf First of all},  let $c:I\to (M^3,g)$ be  an arbitrary integral curve  of the vector field
$e_{1}$ by arc length parameter, i.e.,  $c'(t)=e_1$ along $c$. Since  $\nabla_{ e_1} e_1=0$, we have   $\nabla _{c'(t)}c'(t)=0$ and hence $c$ is a geodesic in $M$
which is horizontal at a neighborhood on $U\subset M$.  Therefore, by
Corollary \ref{Coro3}, $c$ is horizontal throughout and
$\gamma=\pi\circ c$ is a geodesic in $N$. Using the existence of geodesic pole coordinates, we know that there exists a local orthonormal frame   $\{\varepsilon_1, \varepsilon_2\}$ on $N^2$  such that $\varepsilon_1|_\gamma=\gamma'$. Denote by $\varepsilon_1(t),\; \varepsilon_2(t)$  the vector fields $\{\varepsilon_1, \;\varepsilon_2\}$ along $\gamma$, i.e., $\varepsilon_1(t)=\varepsilon_1|_\gamma=\gamma',\; \varepsilon_2(t)=\varepsilon_2|_\gamma$. Let $e_1(t)$ and $e_2(t)$ denote the horizontal vector fields $e_1$ and $e_2$ along $c$, respectively, i.e., $e_1(t)=e_1|_{c(t)}=c'(t)$, $e_2(t)=e_2|_{c(t)}$. It follow that $e_1(t)$,\;$e_2(t)$ along $c$ are $\pi$-related to $\varepsilon_1(t),\; \varepsilon_2(t)$ along $\gamma$, respectively, since $e_1(t)=c'(t)$ is the horizontal lift  of  $\varepsilon_1(t)=\gamma'$ tangent to $\gamma$  by Proposition \ref{Pro3}.\\

 {\bf Secondly},\; since $\alpha_1, \alpha_2$ be the horizontal lifts of  $\eta_1, \eta_2$, we have $\alpha_1(t)=\alpha_1|_{c(t)}, \alpha_2(t)=\alpha_2|_{c(t)}$  along $c$ is $\pi$-related to $\eta_1(t)=\alpha_1|_{\gamma}, \eta_2(t)=\alpha_2|_{\gamma}$ along $\gamma=\pi\circ c$, respectively.
If we assume that $\bar{\theta}$ is the angle between the geodesic curve $\gamma$ and  $u$-curves, then we have
\begin{equation}\label{Z8}\notag
\begin{array}{lll}
\varepsilon_1(t)=\bar{a}\eta_1(t)+\bar{b}\eta_2(t),\;
\varepsilon_2(t)=-\bar{b}\eta_1(t)+\bar{a}\eta_2(t),\;
\bar{a}=\cos\bar{\theta},\;\bar{b}=\sin\bar{\theta}.
\end{array}
\end{equation}
Therefore, using Proposition \ref{Pro4}, we have
\begin{equation}\label{Z9}\notag
\eta_2(t)(\bar{\theta})=H_2|_\gamma,
\end{equation}
where $H_2=-\lambda_u$.
Noting that $\bar{a}=\cos\bar{\theta}, \bar{b}=\sin\bar{\theta}$,  one can compute the following
\begin{equation}\label{Z10}
\begin{array}{lll}
\eta_2(t)(\bar{a})=\eta_2(t)(\cos\bar{\theta})
=-\sin\bar{\theta}\eta_2(t)(\bar{\theta})=-\bar{b}H_2|_\gamma.
\end{array}
\end{equation}
{\bf Thirdly}, we assume that  $\theta$ is the angle between  $e_1(t)$  and  $\alpha_1(t)$, in other words, $\theta$ is the angle between  $e_1$  and  $\alpha_1$ along $c$. Noting that $\bar{\theta}$ is the angle between the geodesic curve $\gamma$ and  $u$-curves, that is,  $\bar{\theta}$ is the angle between  $\varepsilon_1(t)$  and  $\eta_1(t)$, then  we can check that $a=\cos\theta=\langle e_1,\alpha_1\rangle=\langle \varepsilon_1,\eta_1\rangle\circ\pi=\cos\bar{\theta}\circ\pi=\bar{a}\circ\pi$ along\; $c$,\; i.e., $a=\bar{a}\circ\pi$ along $c$,  since a Riemannian submersion preserves the inner product of horizontal vector fields. Similarly, $b=\sin\theta=\sin\bar{\theta}\circ\pi=\bar{b}\circ\pi$ along $c$. It follows that $\theta=\bar{\theta}\circ\pi$ along $c$, i.e., $\theta|_c=\bar{\theta}|_{\pi\circ c}=\bar{\theta}|_\gamma$ along $c$. Along $c$, a straightforward computation using (\ref{Z10}) gives
\begin{equation}\label{Z11}
\begin{array}{lll}
\alpha_2(t)(a)=\alpha_2(t)(\bar{a}\circ\pi)=d\pi[\alpha_2(t)](\bar{a})=\eta_2(t)(\bar{a})=-\bar{b}H_2|_\gamma
\\=-(\bar{b} H_2)\circ\pi\circ c=-bh_2|_c.
\end{array}
\end{equation}
On the other hand, along $c$, by Eq. (\ref{Z7}), we have
\begin{equation}\label{Z12}
\alpha_2(a)=abf_2-bh_2.
\end{equation}
Comparing Eq.(\ref{Z11}) with Eq.(\ref{Z12}) along $c$, we have
\begin{equation}\label{Z13}\notag
abf_2|_c=0.
\end{equation}

{\bf Finally}, suppose  an arbitrary point $p\in M$ and let $c:I\to M^3$ be an arbitrary integral curve of $e_1$ by arc length parameter so that $c(0)=p,\;c'(t)=e_1(t)$,
from the above equation we conclude that the function $abf_2$  vanishes  along $c$. Then, we have  $abf_2|(p)=0,\;\forall p\in M$, that is
\begin{equation}\label{Z14}\notag
abf_2=0,\; {\rm on}\; M.
\end{equation}
As we remark earlier if $ab=0$ then the orthonormal frame $\{e_1, e_2\}$ is obtained from $\{\alpha_1, \alpha_2\}$ by a rotation of $\pi/2$ or $\pi$, therefore  $\{e_1, e_2, e_3\}$ is an adapted frame to the Riemannian submersion. Otherwise, we have $f_2=0$ and hence $\nabla_{e_{2}}e_{2}=f_2e_1=0$. Thus, we complete the proof of the lemma.
\end{proof}

\section{ A complete  classification  of biharmonic Riemannian submersions from BCV 3-spaces}

Note that BCV 3-spaces with $R=4m-l^2=0$ are  space forms $\r^3$ and $S^3$ and that biharmonic Riemannian submersions from a 3-dimensional space form were completely classified in \cite{WO}.

Now we  are ready  to prove the existence of a preferred adapted frame on a BCV space with $R=4m-l^2\neq0$ which is the key ingredient in our proof of the classification theorem.

\begin{theorem}\label{THC}
Let $\pi:M^3_{m,\;l}\to (N^2,h)$ be a Riemannian submersion from BCV 3-spaces with $R=4m-l^2\neq0$. Then, there exists an orthonormal  frame $\{e_1=a_1^1E_1+a_1^2E_2,\; e_2,
\;e_3\}$  adapted to the Riemannian submersion  with $e_3$ vertical.

\end{theorem}

\begin{proof}
It is easy to see that if  the vector field $ E_3$  is tangent to the fibers of the  Riemannian
submersion $\pi$, then $e_3=\pm E_3$ and  any  adapted frame has the required form.\\

From now on, we assume $ e_3\ne \pm E_3$. Then, the vector filed $e_1=\frac{e_3\times E_3}{| e_3\times E_3|}$  is  horizontal  and $\langle e_1, E_3\rangle=0$.
 It is not difficult to see that the frame $\{e_1, \;e_2=e_3\times e_1,\;e_3\}$ is an orthonormal frame on $M^3$.
Let $e_i=\sum\limits_{j=1}^{3}a_i^jE_j, i=1,2,3$ be the transformation between the frames $\{E_i\}$ and $\{e_i\}$. By the choice of $e_1$ we have  $a^3_1=\langle e_1, E_3\rangle=0$, and hence $e_1=a_1^1E_1+a_1^2E_2$ with $(a_1^1)^2+(a_1^2)^2=1$. This implies
\begin{equation}\label{zb}
\begin{array}{lll}
a_1^3=0,\;a_3^{3}\neq\pm1,\;{\rm (and\;hence)}\; a_2^{3}\neq 0.
\end{array}
\end{equation}
We will prove that in  this case,  we have
\begin{equation}\label{bb1}
 f_1=0,\;\nabla_{e_1}e_1=0.
\end{equation}
In fact, a straightforward computation using (\ref{R2}) gives
\begin{equation}\label{bb2}
\begin{array}{lll}
-f_1\sum\limits_{i=1}^{3}a_2^iE_i=-f_1e_{2}=\nabla_{e_{1}} e_{1}=\nabla_{e_{1}}(\sum\limits_{i=1}^{3}a_1^iE_i)
=\sum\limits_{i=1}^{3}e_1(a_1^i)E_i+\sum\limits_{i,j=1}^{3}a_1^ja_1^i\nabla_{E_j}E_i.
\end{array}
\end{equation}
By using (\ref{Lie}), $a_1^3=0$, and comparing  the coefficient of $E_3$ of  (\ref{bb2}) we have $-f_1 a_2^3=e_1(a_1^3)=0$, which gives $f_1=0$ since $a_2^3\neq0$. \\
Using  (\ref{BCV1}), (\ref{R2}) with  $a_1^3=f_1=0$,  and a further computation  similar to  those used  in computing  (\ref{bb2}) we obtain
\begin{equation}\label{thb2}
\begin{array}{lll}
e_1(a_{2}^{3})=-(\sigma+\frac{l}{2})a_{3}^{3},\;
e_1(a_{3}^{3})=(\sigma+\frac{l}{2})a_{2}^{3},\;
e_2(a_{2}^{3})=0,\;
e_2(a_{3}^{3})=0,\\
e_3(a_{2}^{3})=-\kappa_2a_{3}^{3},\;
e_3(a_{3}^{3})=\kappa_2a_{2}^{3},\;
\kappa_1a_{3}^{3}=(\sigma-\frac{l}{2}-f_{3})a_{2}^{3},\;
f_2a_{2}^{3}=(\sigma+\frac{l}{2})a_{3}^{3}.
\end{array}
\end{equation}

Since $\nabla_{e_{1}}e_{1}=0$, we use Lemma \ref{L1} to conclude that $\nabla_{e_{2}}e_{2}\neq0$, and the chosen orthonormal
frame $\{e_1, e_2, e_3\}$  is an adapted frame to the Riemannian submersion $\pi$, or \;$\nabla_{e_{2}}e_{2}=0$.
So, we only need to consider the latter case: \;$\nabla_{e_{2}}e_{2}=0$, i.e., $f_2=0$. Together with $f_1=a_1^3=0$ and a straightforward computation using (\ref{BCV1}), (\ref{BCV2}), (\ref{R2})\; and \;(\ref{RC}), we have the following equations
\begin{equation}\label{RC1}
\begin{cases}
-e_1(\sigma)+2\kappa_1\sigma=-a_{2}^{3}a_{3}^{3}R,\\
e_1(\kappa_1)+\sigma^2-\kappa_{1}^2=(a_{2}^{3})^2R+\frac{l^2}{4},\;\\
e_1(\kappa_2)-e_3(\sigma)-\kappa_1\kappa_2=0,\;\\
2f_{3}\sigma-3\sigma^2=(a_{3}^{3})^2R+\frac{l^2}{4},\\
-e_2(\sigma)+2\kappa_2\sigma=0,\\
e_2(\kappa_{1})+e_3(\sigma)-\kappa_1 \kappa_2=0,\\
\sigma^{2}+e_2(\kappa_2)-\kappa_2^2=\frac{l^2}{4},\\
\end{cases}
\end{equation}
where $R=4m-l^2\neq0$.\\
Firstly, we show that $\sigma=-\frac{l}{2},\;\kappa_2=0$, and\;$\kappa_1\neq0$. Moreover, $m\neq0$ and $l=0$ holds iff $a_3^3=0$ holds. In fact, we use the 2nd, the 8th equation of (\ref{thb2}) and (\ref{zb}) with $f_2=0$ to obtain
$\sigma=-\frac{l}{2}$. Using this, together with the 1st,  the 5th equations of (\ref{RC1}) and the 6th  of (\ref{thb2}), we have $\kappa_2=0$. Since $(a_2^3)^2R\neq0$,
we apply the 2nd equation of (\ref{RC1}) with $\sigma=-\frac{l}{2}$ to conclude that $\kappa_1\neq0$. From these and the 1st equation of (\ref{RC1}), we can check that  $m\neq 0$ and $l=0$ is equivalent to $a_3^3=0$. \\

Based on the above, we only need to prove that $f_3=0$ and hence the frame $\{e_1, e_2, e_3\}$ is adapted to the Riemannian submersion in the following two cases.\\

Case I: $a_3^3=a_1^3=f_1=f_2=0$,  and hence $a_2^3=\pm1$ and $l=\sigma=0$. In this case, we have $f_3=0$ immediately from the 7th equation of (\ref{thb2}).\\

Case II: $a_3^3\neq0,\pm1$, $a_1^3=f_1=f_2=0$,\;$\kappa_1\sigma l\neq0$,  and $a_2^3\neq0,\pm1$. In this case,  we use  $\sigma=-\frac{l}{2}$, $\kappa_2=0$,  the 1st, the 2nd, the 5th  and the 6th equation of (\ref{thb2}) separately to have $e_1(a_2^3)=e_1(a_3^3)=e_3(a_2^3)=e_3(a_3^3)=0$. From these, together with the 3rd  and the 4th equation of (\ref{thb2}),  we conclude that $a_2^3, a_3^3$ are constants. Using these, the 1st and the 4th equations  of (\ref{RC1}), we  see that both $\kappa_1$ and $f_3$ are constants. Substituting these and $\sigma=-\frac{l}{2}$ into the 1st and  2nd equation of (\ref{RC1}) we have $\kappa_1l=a_2^3a_3^3R,\;{\rm and}\;\kappa_1^2=-(a_2^3)^2R$. It follows that $(a_3^3)^2=\frac{-l^2}{4m-l^2}$, and hence  $(a_2^3)^2=1-(a_3^3)^2=\frac{4m}{4m-l^2}$ and  $\kappa_1^2=-4m$ being positive constants. From these and  the 4th equation of (\ref{RC1}), we  have $f_3=0$.\\

 Combining the results obtained in Cases I, II,  and those obtained  in the case of $\nabla_{e_2}e_2\ne 0$ we obtain the theorem.
\end{proof}
Now we are ready to prove our main theorem.
\begin{theorem}\label{Th1}
A proper biharmonic Riemannian submersion $\pi:M^3_{m,\;l}\to (N^2,h)$
from a BCV 3-space  exists only  in the cases: $H^2(4m)\times\r
\to \r^2$,\;or\;$\widetilde{SL}(2,\r)\to\r^2$.
\end{theorem}
\begin{proof}
 First of all, we may assume that  $R=4m-l^2\neq0$  for the reason mentioned at the beginning of the section. We also assume that $a_3^{3}\neq\pm1$ for otherwise one can  use (\ref{BCV1}) to check that the tension of the Riemannian  submersion  $\tau(\pi)=- d\pi(\nabla^{M}_{E_3}E_3)=0$, and hence it is harmonic. \\

 Adopting the same notations and sign convention as in the proof of Theorem \ref{THC}, we know from Theorem \ref{THC} that there exists an orthonormal frame $\{e_1,\; e_2,
\;e_3\}$ adapted to the Riemannian submersion $\pi$ with $e_1=a_1^1E_1+a_1^2E_2$, $e_3$ being  vertical and hence $a_1^3=f_1=f_3=0$, $e_3(f_1)=e_3(f_2)=0$ (see \cite{WO}). These reduce (\ref{RC})  into
\begin{equation}\label{RC2}
\begin{cases}
-e_1(\sigma)+2\kappa_1\sigma=-a_{2}^{3}a_{3}^{3}R,\\
e_1(\kappa_1)+\sigma^2-\kappa_{1}^2=(a_{2}^{3})^2R+\frac{l^2}{4},\;\\
e_1(\kappa_2)-e_3(\sigma)-\kappa_1\kappa_2=0,\;\\
e_1(f_2)-f_{2}^{2}-3\sigma^2=(a_{3}^{3})^2R+\frac{l^2}{4},\\
-e_2(\sigma)+2\kappa_2\sigma=0,\\
e_2(\kappa_{1})+e_3(\sigma)+\kappa_2 f_2-\kappa_1 \kappa_2=0,\\
\sigma^{2}+e_2(\kappa_2)-\kappa_1f_2-\kappa_2^2=\frac{l^2}{4},\\
\end{cases}
\end{equation}
where $R=4m-l^2\neq0.$\\

According to Theorem \ref{THC}, we only need to consider the two cases: $\nabla_{e_2}e_2=0$ and $\nabla_{e_2}e_2\ne 0$. \\

For the case of  $\nabla_{e_2}e_2=0$, as in the proof of Theorem \ref{THC}, we have the following two sub-cases: Case I and Case II.\\

{\bf Case I:} $a_3^3=a_1^3=f_1=f_2=f_3=l=\kappa_2=\sigma=0$, $a_2^3=\pm 1$, $\kappa_1\neq0$, $R=4m-l^2=4m\neq0$, and $(N^2,h)$ is  (by  (\ref{GCB}))  locally $\r^2$. It is easily checked that in this case, (\ref{RC2}) reduces to
\begin{equation}\label{th3}
e_1(\kappa_1)=\kappa_1^2+4m,\;e_2(\kappa_1)=0,
\end{equation}
and biharmonic equation (\ref{lem1}) reads
\begin{equation}\label{th4}
\Delta\kappa_1=0.
\end{equation}
A straightforward computation gives
\begin{equation}\label{th5}
\begin{array}{lll}
\Delta\kappa_1=e_1e_1(\kappa_1)+e_3e_3(\kappa_1)-\nabla_{e_1}{e_1}(\kappa_1)-\nabla_{e_2}{e_2}(\kappa_1)-\nabla_{e_3}{e_3}(\kappa_1)\\
=e_1(\kappa_1^2+4m)+e_3e_3(\kappa_1)-\kappa_1e_1(\kappa_1)=e_3e_3(\kappa_1)+\kappa_1^3+4m\kappa_1.
\end{array}
\end{equation}
Substituting  (\ref{th5}) into (\ref{th4}), we have
\begin{equation}\label{th6}
\begin{array}{lll}
e_3e_3(\kappa_1)=-\kappa_1^3-4m\kappa_1.
\end{array}
\end{equation}
Applying $e_3$ to both sides of  the 1st equation of (\ref{th3}) and using the fact that $e_1e_3(\kappa_1) = [e_1, e_3](\kappa_1) + e_3e_1(\kappa_1)$,
we have
\begin{equation}\label{th7}
\begin{array}{lll}
e_1e_3(\kappa_1)=3\kappa_1e_3(\kappa_1).
\end{array}
\end{equation}
Using  (\ref{th3}), (\ref{th6}),  (\ref{th7}), and a direct computation we get
\begin{equation}\label{th8}
\begin{array}{lll}
e_1e_3\{e_3(\kappa_1)\}-e_3e_1\{e_3(\kappa_1)\} =[e_1,e_3] \{e_3(\kappa_1)\}
=\kappa_1e_3e_3(\kappa_1)
=-\kappa_1^4-4m\kappa_1^2,
\end{array}
\end{equation}
and
\begin{equation}\label{th9}
\begin{array}{lll}
e_1e_3\{e_3(\kappa_1)\}-e_3e_1\{e_3(\kappa_1)\}
=e_1\{e_3e_3(\kappa_1)\}-e_3\{e_1e_3(\kappa_1)\}\\
=-4m\kappa_1^2-16m-3e_3^2(\kappa_1).\\
\end{array}
\end{equation}

Comparing  (\ref{th8}) with (\ref{th9}), we get
\begin{equation}\label{th10}
\begin{array}{lll}
3e_3^2(\kappa_1)=\kappa_1^4-16m.\\
\end{array}
\end{equation}
 Applying $e_3$ to both sides of  (\ref{th10}) and using (\ref{th6}) to
simplify the resulting equation we have
\begin{equation}\label{th12}
\begin{array}{lll}
\kappa_1(5\kappa_1^2+12m)e_3(\kappa_1)=0,
\end{array}
\end{equation}
which implies $e_3(\kappa_1)=0$. Substituting this into (\ref{th6}) and using that fact that  $m\kappa_1\neq0$ we obtain $\kappa_1^2=-4m>0$. Therefore, the BCV space in this case has $m<0\;{\rm and}\;l=0$, so it is $H^2(4m)\times\r$. Thus, a potential proper biharmonic Riemannian submersion exists  in the case $H^2(4m)\times\r\to\r^2$. \\

{\bf Case II:} $a_2^3,\;a_3^3\neq0,\pm1$, $a_1^3=f_1=f_2=f_3=\kappa_2=0$,\;$\kappa_1\neq0$, $\sigma=-\frac{l}{2}\neq0$, and $(N^2,h)$ is (by  (\ref{GCB}))  locally $\r^2$.
In this case, as in the proof the Theorem \ref{THC}, we have $\kappa_1^2=-4m>0$, $(a_3^3)^2=\frac{-l^2}{4m-l^2},$ and $(a_2^3)^2=\frac{4m}{4m-l^2}$ are positive constants. It
is easily checked that the biharmonic equation  (\ref{lem1}) holds.  Note that in this case, we have $m<0\;{\rm and}\;l\neq0$, so the BCV space is  $\widetilde{SL}(2,\r)$, and  the potential  proper biharmonic Riemannian submersion  exists  in the case of  $\widetilde{SL}(2,\r)\to\r^2$. \\

For the case of $\nabla_{e_2}e_2\ne 0$, we will show that there exists no proper biharmonic Riemannian submersion from a BCV space.\\

Note that in this case the hypotheses can be summarized as:
\begin{align} \label{H}
a_3^3\neq0,\;\pm1, f_1=f_3=a_1^3=0, f_2\neq0, R=4m-l^2\neq0, a_2^3\neq0,\;\pm1
\end{align}

{\bf Claim I:}  Under the hypotheses (\ref{H}), we have
\begin{align}\label{step1}
e_2(f_2)=e_2(\kappa_1)=e_3(\kappa_1)=e_2(\sigma)=e_3(\sigma)=\kappa_2=0,\;\kappa_1\neq0,\;{\rm and}\;\sigma\neq0,\pm\frac{l}{2}.
\end{align}

 {\bf Proof of Claim I:} (i)  $\sigma\neq0$ since if $\sigma=0$, we use the 1st equation of (\ref{RC2}) to have $a_2^3a_3^3R=0$ a contradiction.\\

 (ii) $\kappa_2=e_2(\kappa_1)=e_2(\sigma)=e_3(\sigma)=0$.  In fact, applying $e_3$ to both sides of the 4th
 equation of (\ref{RC2}) and the 8th equation of (\ref{thb2}) separately and using the 5th, the 6th  equations of  (\ref{thb2}), (\ref{GCB0}), $f_3=0$ and $e_3(f_1)=e_3(f_2)=0$, we obtain
\begin{equation}\label{th15}
\begin{array}{lll}
e_3(\sigma)=-\frac{1}{3\sigma}\kappa_2a_2^3a_3^3R,\;\;
-f_2\kappa_2a_3^3=e_3(\sigma)a_3^3+(\sigma+\frac{l}{2})\kappa_2a_2^3.\\
\end{array}
\end{equation}
Using (\ref{th15}) and the 8th equation of (\ref{thb2}) we have
\begin{equation}\label{th17}
\begin{array}{lll}
\kappa_2\left(-(a_2^3a_3^3)^2R+3\sigma(\sigma+\frac{l}{2})\right)=0,\\
\end{array}
\end{equation}
which implies $\kappa_2=0$,\;or
\begin{equation}\label{th20}
\begin{array}{lll}
3\sigma(\sigma+\frac{l}{2})=(a_2^3a_3^3)^2R.
\end{array}
\end{equation}
For the latter case, applying $e_2$ to both sides of (\ref{th20}) and using the 3rd and the 4th equation of (\ref{thb2}) yields
\begin{equation}\label{th21}
\begin{array}{lll}
3(2\sigma+\frac{l}{2})e_2(\sigma)=0,
\end{array}
\end{equation}
which implies $e_2(\sigma)=0.$
Substituting this into the 5th equation of (\ref{RC2}) and using the fact that $\sigma\neq0$ we also have $\kappa_2=0$,  and hence $e_3(\sigma)=0$ by (\ref{th15}). From these and  the 6th equation of (\ref{RC2}), we get $e_2(\kappa_1)=0$.\\

(iii) $e_2(f_2)=e_3(\kappa_1)=0$. Indeed, applying $e_2$ to
both sides of the 2nd equation of  (\ref{RC2}) and the 8th
equation of  (\ref{thb2}) separately and using the 3rd and the 4th  equations of  (\ref{thb2}) together with $a_2^3\neq0$,\; $e_2(\kappa_1)=e_2(\sigma)=0$\;and\;$e_2(a_2^3)=0$, we get
\begin{equation}\label{th24}
\begin{array}{lll}
e_2e_1(\kappa_1)=0,\;e_2(f_2)=0,\;e_1e_2(\kappa_1)-e_2e_1(\kappa_1)=0.
\end{array}
\end{equation}
These, together with $0=e_1e_2(\kappa_1)-e_2e_1(\kappa_1)= [e_1,
e_2](\kappa_1)=-2\sigma e_3(\kappa_1)$\;and\;$\sigma\neq0$,  imply that $e_3(\kappa_1)=0$.\\

(iv)  $\sigma\neq\pm\frac{l}{2}$ and $\kappa_1\neq0$. Indeed, substituting $\kappa_2=0$ into  the 7th equation
of (\ref{RC2}) we obtain
\begin{equation}\label{th25}
\begin{array}{lll}
\kappa_1f_2=\sigma^2-\frac{l^2}{4},
\end{array}
\end{equation}
which, together with $f_2\neq0$, implies that $\sigma=\pm\frac{l}{2}$ is equivalent to $\kappa_1=0$. Clearly, if $\sigma=\pm \frac{l}{2}$ and hence $\kappa_1=0$,  then the 2nd equation of  (\ref{RC2}) implies that $(a_2^3)^2R=0$,  which is a contradiction since $a_2^3\neq0$ and $R\neq0$. Thus, we have $\sigma\neq\pm\frac{l}{2}$ and $\kappa_1\neq0$, which completes the proof of Claim I.\\

 {\bf Claim II}: Under the same hypotheses (\ref{H}), we have $\sigma =-\frac{l}{2}$.\\

 {\bf Proof of Claim II:} Since $f_1=\kappa_2=0$ and $K^N=e_1(f_2)-f_2^2$, the biharmonic equation (\ref{lem1}) reduces to
\begin{equation}\label{thb26}
\begin{array}{lll}
\Delta\kappa_1-\kappa_1\{-e_1(f_2)+2f_2^2\}=0.
\end{array}
\end{equation}
Applying $e_1$ to both sides of the 2nd
equation of  (\ref{RC2}) and (\ref{th25}) separately and using the the 1st equation of (\ref{thb2}), we have
\begin{equation}\label{th26}
\begin{array}{lll}
e_1e_1(\kappa_1)=2\kappa_1e_1(\kappa_1)-2\sigma e_1(\sigma)-2(\sigma+\frac{l}{2})a_2^3a_3^3R,\\
\kappa_1e_1(f_2)+f_2e_1(\kappa_1)=2\sigma e_1(\sigma).
\end{array}
\end{equation}
A straightforward computation using the 1st, the 2nd equations of (\ref{RC2}), the 8th equation of (\ref{thb2}), (\ref{th25}), (\ref{th26}), (\ref{step1}), (\ref{R2}) and (\ref{thb26})  with
$f_1=0$ yields
\begin{equation}\label{th27}
\begin{array}{lll}
0=\Delta\kappa_1-\kappa_1\{-e_1(f_2)+2f_2^2\}\\
=e_ie_i(\kappa_1)-\nabla_{e_i}e_i(\kappa_1)-\kappa_1\{-e_1(f_2)+2f_2^2\}\\
=\kappa_1^3-3\kappa_1^2f_2+\kappa_1(a_2^3)^2R-4(\sigma
+\frac{l}{2})a_2^3a_3^3R.
\end{array}
\end{equation}

Note that $f_3=0$, so the 7th equation of (\ref{thb2}) becomes
$\kappa_1a_3^3=(\sigma-\frac{l}{2})a_2^3$.
 By multiplying $a_3^3$ to both sides of (\ref{th27})  and using the fact that
$\kappa_1a_3^3=(\sigma-\frac{l}{2})a_2^3$,\;
$\kappa_1f_2=\sigma^2-\frac{l^2}{4}$ \; and $(a_2^3)^2+(a_3^3)^2=1$\;to simplify the resulting
equation,  we get
\begin{equation}\label{th28}
\begin{array}{lll}
\kappa_1^2=\frac{5\sigma+\frac{3l}{2}}{\sigma-\frac{l}{2}}R(a_3^3)^2+3\sigma^2-\frac{3l^2}{4}-R.
\end{array}
\end{equation}
Similarly,  multiplying $(a_3^3)^2$ to both  sides of (\ref{th28}), using $\kappa_1a_3^3=(\sigma-\frac{l}{2})a_2^3$,  $(a_2^3)^2+(a_3^3)^2=1$, and simplifying the resulting
equation yields
\begin{equation}\label{th29}
\begin{array}{lll}
\frac{5\sigma+\frac{3l}{2}}{\sigma-\frac{l}{2}}R(a_3^3)^4+(4\sigma^2-\sigma
l-\frac{l^2}{2}-R)(a_3^3)^2-(\sigma-\frac{l}{2})^2=0.
\end{array}
\end{equation}
 On the other hand, by applying  $e_1$ to both sides of (\ref{th28}) and using the 1st, the 2nd equations of (\ref{RC2}), the 1st and the 2nd equations of (\ref{thb2}) we have \begin{equation}\label{thb30}
\begin{array}{lll}
2\kappa_1\{\kappa_1^2+(a_2^3)^2R-\sigma^2+\frac{l^2}{4}\}=-\frac{4le_1(\sigma)}{(\sigma-\frac{l}{2})^2}R(a_3^3)^2+\frac{5\sigma+\frac{3l}{2}}{\sigma-\frac{l}{2}}2Ra_3^3e_1(a_3^3)
+6\sigma e_1(\sigma)\\
=-\frac{8l\kappa_1\sigma}{(\sigma-\frac{l}{2})^2}R(a_3^3)^2-\frac{4lR^2}{(\sigma-\frac{l}{2})^2}a_2^3(a_3^3)^3
+\frac{2(5\sigma+\frac{3l}{2})(\sigma+\frac{l}{2})}{\sigma-\frac{l}{2}}Ra_3^3a_2^3
+12\kappa_1\sigma^2+6\sigma a_2^2a_3^3R,
\end{array}
\end{equation}
which can be simplified as
\begin{equation}\label{th30}
\begin{array}{lll}
\kappa_1^3=\kappa_1(7\sigma^2-\frac{l^2}{4})-\kappa_1(a_2^3)^2R-\frac{4\kappa_1\sigma
lR}{(\sigma-\frac{l}{2})^2}(a_3^3)^2-\frac{2l
R^2}{(\sigma-\frac{l}{2})^2}a_2^3(a_3^3)^3\\+
\frac{(5\sigma+\frac{3l}{2})(\sigma+\frac{l}{2})}{\sigma-\frac{l}{2}}Ra_2^3a_3^3+3\sigma a_2^3a_3^3R.
\end{array}
\end{equation}
Multiplying $a_3^3$ to both  sides of (\ref{th30}),
using the fact that $\kappa_1a_3^3=(\sigma-\frac{l}{2})a_2^3$,  $(a_2^3)^2+(a_3^3)^2=1$, and
simplifying the resulting equation gives
\begin{equation}\label{th31}
\begin{array}{lll}
\kappa_1^2=-\frac{2l
R^2}{(\sigma-\frac{l}{2})^3}(a_3^3)^4+\frac{(9\sigma^2-\frac{5\sigma l}{2}+
l^2)R}{(\sigma-\frac{l}{2})^2}(a_3^3)^2+7\sigma^2-\frac{l^2}{4}-R.
\end{array}
\end{equation}
Combining  (\ref{th28}) and  (\ref{th31}) yields
\begin{equation}\label{th32}
\begin{array}{lll}
-\frac{2l R^2}{(\sigma-\frac{l}{2})^3}(a_3^3)^4+\frac{(4\sigma^2-\frac{3\sigma l}{2}+
\frac{7l^2}{4})R}{(\sigma-\frac{l}{2})^2}(a_3^3)^2+2(2\sigma^2+\frac{l^2}{4})=0.
\end{array}
\end{equation}
Adding a $ (5\sigma+\frac{3l}{2})(\sigma-\frac{l}{2})^2$ multiple of (\ref{th32}) to a $(2lR)$ multiple of
(\ref{th29}) and simplifying the results yields

\begin{equation}\label{th33}
\begin{array}{lll}
\left((5\sigma+\frac{3l}{2})(4\sigma^2-\frac{3\sigma l}{2}+\frac{7l^2}{4})+2l(4\sigma^2-\sigma
l-\frac{l^2}{2}-R)\right)R(a_3^3)^2\\=2(\sigma-\frac{l}{2})^2\left(-(5\sigma+\frac{3l}{2})(2\sigma^2+\frac{l^2}{4})+lR\right).
\end{array}
\end{equation}
Similarly, adding a $ (\sigma-\frac{l}{2})^2$ multiple of  (\ref{th32}) to a $2(2\sigma^2+\frac{l^2}{4})$  multiple (\ref{th29}) and simplifying the results  gives
\begin{equation}\label{th34}
\begin{array}{lll}
(\sigma-\frac{l}{2})\left((4\sigma^2-\frac{3\sigma l}{2}+\frac{7l^2}{4})R+2(4\sigma^2-\sigma
l-\frac{l^2}{2}-R)(2\sigma^2+\frac{l^2}{4})\right)\\=2\left(-(5\sigma+\frac{3l}{2})(2\sigma^2+\frac{l^2}{4})+lR\right)R(a_3^3)^2.
\end{array}
\end{equation}
Combining (\ref{th33}) and  (\ref{th34}) to eliminate $a_3^3$ we obtain
\begin{equation}\label{th35}
\begin{array}{lll}
\left((5\sigma+\frac{3l}{2})(4\sigma^2-\frac{3\sigma l}{2}+\frac{7l^2}{4})+2l(4\sigma^2-\sigma
l-\frac{l^2}{2}-R)\right)\\
\times \left((4\sigma^2-\frac{3\sigma l}{2}+\frac{7l^2}{4})R+2(4\sigma^2-\sigma
l-\frac{l^2}{2}-R)(2\sigma^2+\frac{l^2}{4})\right)\\
=4(\sigma-\frac{l}{2})\left(-(5\sigma+\frac{3l}{2})(2\sigma^2+\frac{l^2}{4})+lR\right)^2.
\end{array}
\end{equation}
Noting that (\ref{th35}) can be written as $80 \sigma^7+ p(\sigma)=0$, where $p(\sigma)$ denotes a polynomial in $\sigma$ of degree less $\le 6$.

It follows that $\sigma$ is a constant if it exists. This, together with (\ref{th32}) and (\ref{th31}), implies that $a_3^3$ and $\kappa_1$ are  constants, and hence $a_2^3=\pm\sqrt{1-(a_3^3)^2}$ is also constant.  Using these, the 1st equation of
(\ref{thb2}), and $a_3^3\neq0$, we have $\sigma=-\frac{l}{2}$, which completes the proof of Claim II.

Clearly, Claim II contradicts  $\sigma\neq - \frac{l}{2}$ in Claim I.  Thus, we conclude that there exists no proper biharmonic Riemannian submersion in the case of $\nabla_{e_2}e_2\ne 0$ and $a_3^3\neq0, \pm1$.\\
Summarizing all above results, we obtain the theorem.
 \end{proof}
 
  Using Theorem 3.3 in \cite{WO} and  Theorem \ref{Th1} above, we immediately have
\begin{corollary} \label{Coro4}
 A Riemannian submersion from $S^3$, $\r^3$, or, $S^2 \times\r$, or, $SU(2)$, or, Nil space $($Heisenberg space$)$  to a surface is biharmonic if and only if it is harmonic.
\end{corollary}
From the proof of Theorem  \ref{Th1}, if $a_1^3=f_1=0$, $a_3^3=\langle e_3, E_3\rangle\neq0, \pm1$  and the integrability data $f_2\neq0$,  then there exists no biharmonic Riemannian submersion from BCV 3-spaces.
\begin{example}\label{ex0}
The Riemannian submersion from Heisenberg group (Nil space)
\begin{align}\notag
\pi  : (\r^3 , {\rm d}x^{2}+{\rm
d}y^{2}+({\rm d}z+\frac{1}{2}(y{\rm d}x-x{\rm d}y)^{2})) &\to (\r^2 ,dx^2 + (1+x^2)^{-1}dt^2)
\\\notag
\pi(x,y,z) =(x,z+\frac{xy}{2})
\end{align}
is not a biharmonic map.
\end{example}
For $m=0\;{\rm and}\;l=1$, the BCV 3-space is a Heisenberg group (Nil space): $(\r^3 , g={\rm d}x^{2}+{\rm
d}y^{2}+\left({\rm d}z+\frac{1}{2}(y{\rm d}x-x{\rm d}y\right)^{2})$ with an  orthonormal frame $\{E_1=\frac{\partial}{\partial
x}-\frac{y}{2}\frac{\partial}{\partial z},\;E_2=\frac{\partial}{\partial
y}+\frac{x}{2}\frac{\partial}{\partial z}, \;
E_3=\frac{\partial}{\partial z}\}$. We can check that $e_1=E_1,\;e_2=-\frac{x}{\sqrt{1+x^2}}E_2-\frac{1}{\sqrt{1+x^2}}E_3, \;
e_3=\frac{1}{\sqrt{1+x^2}}E_2-\frac{x}{\sqrt{1+x^2}}E_3$ form an
orthonormal frame on Nil space adapted to the Riemannian submersion with $d\pi(e_1)=\varepsilon_1,\;d\pi(e_2)=\varepsilon_2$, where $\varepsilon_1=\frac{\partial}{\partial x},\;\varepsilon_2=-\sqrt{1+x^2}\frac{\partial}{\partial t}$ and $e_3$ being vertical. Using (\ref{BCV1}), the Lie brackets given by
\begin{align*}
[e_1,e_2]=\frac{x}{1+x^2}e_2-\frac{1-x^2}{1+x^2}e_3,\;
[e_1,e_3]=-\frac{x}{1+x^2}e_3,\; \;[e_2,e_3]=0,
\end{align*}
from which we obtain the integrability data of the Riemannian
submersion $\pi$ as
$f_1=0,\;f_2=\frac{x}{1+x^2},\;\;\kappa_1=-\frac{x}{1+x^2},\;\;\sigma=\frac{1-x^2}{2(1+x^2)},\;\;\kappa_2=0$. It is clear to find that $a_1^3=0,\;f_1=0, a_3^3=\langle e_3, E_3\rangle=-\frac{x}{\sqrt{1+x^2}}\neq0,\pm1$ and $f_2=\frac{x}{1+x^2}\not\equiv0$.  Applying Theorem \ref{Th1}, the Riemannian submersion
$\pi$ is not biharmonic.\\

\section{Some examples of proper biharmonic Riemannian submersions from $H^2\times\r$ and
$\widetilde{SL}(2,\r)$}
Theorem \ref{Th1} shows that  a proper biharmonic Riemannian submersion $\pi:M^3_{m,\;l}\to (N^2,h)$
from a BCV 3-space exists only in the cases of $H^2\times\r
\to \r^2$\;or\; $\widetilde{SL}(2,\r)\to\r^2$. In this section, we
construct many examples of proper biharmonic Riemannian submersions from $H^2\times\r$ and
$\widetilde{SL}(2,\r)$.\\

\begin{proposition}\label{Pr1}
The Riemannian submersion
\begin{align}\label{pr1}\notag
\pi:(\r^{3}_{+},\frac{a^2(dx^2+dy^2)}{y^2}+(dz+\frac{ba}{y} dx)^2)
&\to (\r^2 ,\frac{a^2dy^2}{y^2} + \frac{dz^2}{1+b^2}),\; \pi(x,y,z) =(y,z)
\end{align}
is a proper biharmonic map, where $a>0,\; b\in \r$ are constants and $\r^{3}_{+}=\{(x,y,z)\in \r^{3}, y>0\}$.
\end{proposition}
\begin{proof}
It is not difficult to check that the orthonormal frame
$\{e_1=\frac{y}{a}\frac{\partial}{\partial y},\;e_2= \sqrt{1+b^2}\frac{\partial}{\partial
z}-\frac{by}{a\sqrt{1+b^2}}\frac{\partial}{\partial
x}, \;e_3=\frac{y}{a\sqrt{1+b^2}} \frac{\partial}{\partial x}\}$  on
$(\r^{3}_{+},\frac{a^2}{y^2}(dx^2+dy^2)+(dz+b\frac{a}{y} dx)^2)$ is
adapted to the Riemannian submersion $\pi$ with $
d\pi(e_i)=\varepsilon_i, i=1, 2$ and $e_3$ being vertical, where
$\varepsilon_1=\frac{y}{a}\frac{\partial}{\partial y},\;\varepsilon_2=
\sqrt{1+b^2}\frac{\partial}{\partial
z},\;$ form an orthonormal frame on the
base space $(\r^2 ,\frac{a^2}{y^2}dy^2 + \frac{1}{1+b^2}dz^2)$. It is easy to compute the Lie brackets and the associated the integrability data as follows
\begin{equation}\notag
\begin{array}{lll}
[e_1,e_3]=\frac{1}{a}\; e_3,\;[e_2,e_3]=0 ,\;\; [e_1,e_2]=-\frac{b}{a}e_3,\;
f_1=f_2=\kappa_2=0,\;\kappa_1=\frac{1}{a},\;\sigma=\frac{b}{2a}.
\end{array}
\end{equation}
Substituting these and the curvature $K^{\r^2}=0$ into
biharmonic Equation (\ref{lem1}) we conclude that the Riemannian submersion
$\pi$ is a proper biharmonic map.  From which, we obtain the
proposition.
\end{proof}
It is known that the product space $H^2\times\r$ can be identified with $(\r^{3}_{+},\frac{1}{y^2}(dx^2+dy^2)+dz^2)$, or, $(\r^{3},e^{\frac{2y}{a}}dx^2+dy^2+dz^2)$, or, $(\r^{3},\frac{dx^2+dy^2}{(1-\frac{1}{4a^2}(x^2+y^2))^2}+dz^2)$. As an application of Proposition \ref{Pr1}, for $b=0$, we have

\begin{corollary}
For constant $a>0$, the Riemannian submersion
\begin{align}\label{ex0}
\pi^a :H^2(-\frac{1}{a^2})\times\r=&(\r^{3}_{+},\frac{a^2(dx^2+dy^2)}{y^2}+dz^2)
\to (\r^2 ,\frac{a^2dy^2}{y^2} + dz^2),\; \\\notag & \pi^a(x,y,z) =(y,z)
\end{align}
is a proper biharmonic map from $H^2(-\frac{1}{a^2})\times\r$ to Euclidean space $\r^2$. In particular, for $a=1$, the
Riemannian submersion
\begin{align}\notag
\pi^{1}
:H^2(-1)\times\r=(\r^{3}_{+},\frac{dx^2+dy^2}{y^2}+dz^2)
&\to (\r^2 ,\frac{dy^2}{y^2} + dz^2),\; \notag \pi^1(x,y,z) =(y,z)
\end{align}
is a proper biharmonic map.
\end{corollary}
We can construct proper  biharmonic  Riemannian submersion from the other model  $(\r^{3},g=e^{2\sqrt{-4m}y}dx^2+dy^2+dz^2)$ (which is  actually $H^2(4m)\times\r$)

\begin{example}\label{Ex1}
For $m<0$, the Riemannian submersion
\begin{align}\label{ex1}\notag
\pi_{4m} :H^2(4m)\times\r=&(\r^{3},e^{2\sqrt{-4m}y}dx^2+dy^2+dz^2)
\to (\r^2,dy^2 + dz^2),\; \\ & \pi _{4m}(x,y,z) =(y,z)
\end{align}
is a proper biharmonic map from $H^2(4m)\times\r$ to a Euclidean space
$\r^2$. In particular, for $m=-\frac{1}{4}$, the Riemannian
submersion
\begin{align}\notag
\pi_{-1} :H^2(-1)\times\r=(\r^{3},e^{2y}dx^2+dy^2+dz^2) &\to
(\r^2 ,dy^2 + dz^2),\; \notag \\ \pi_{-1}(x,y,z) =(y,z)
\end{align}
is a proper biharmonic map.
\end{example}
In fact, it is easy to check that the orthonormal frame
$\{e_1=\frac{\partial}{\partial y},\;e_2= \frac{\partial}{\partial
z}, \;e_3=e^{-\sqrt{-4m}y} \frac{\partial}{\partial x}\}$  on
$(\r^{3},g=e^{2\sqrt{-4m}y}dx^2+dy^2+dz^2)$ is adapted to the
Riemannian submersion $\pi$ with $ d\pi(e_i)=\varepsilon_i, i=1, 2$
and $e_3$ being vertical, where
$\varepsilon_1=\frac{\partial}{\partial y},\;\varepsilon_2=
\frac{\partial}{\partial z},\;$ form an orthonormal frame on the
base space $(\r^2 , dy^2 + dz^2)$. A straightforward
computation gives the Lie brackets and the associated the integrability data as
\begin{equation}\notag
\begin{array}{lll}
[e_1,e_3]=-\sqrt{-4m}\; e_3,\;[e_2,e_3]=0 ,\;\; [e_1,e_2]=0,\;
f_1=f_2=\sigma=\kappa_2=0,\;\kappa_1=-\sqrt{-4m}.
\end{array}
\end{equation}
Substituting these and the curvature $K^{\r^2}=0$ into
(\ref{lem1}) we conclude that the Riemannian submersion
$\pi_{4m}$ is a proper biharmonic map. From which, we obtain the
example.\\

Following \cite{BW1} (page 301) we identify  $\widetilde{SL}(2,\r)$ with $(\r^{3}_{+},\frac{dx^2+dy^2}{y^2}+(dz+\frac{1}{y} dx)^2)$ and use Proposition \ref{Pr1} with  $a=b=1$ to have
\begin{corollary}\label{Co}
The Riemannian submersion
\begin{align}\label{coro2}
\pi:\widetilde{SL}(2,\r)=(\r^{3}_{+},\frac{dx^2+dy^2}{y^2}+(dz+\frac{dx}{y} )^2)
&\to (\r^2 ,\frac{dy^2}{y^2} +\frac{dz^2}{2}),\; \notag \pi(x,y,z) =(y,z)
\end{align}
is a proper biharmonic map  from  $\widetilde{SL}(2,\r)$ to a Euclidean space $\r^2$ .
\end{corollary}

\noindent{\bf Declarations:}\\

1. Funding: Ze-Ping Wang was supported by the Natural Science Foundation of China (No. 11861022), Y.-L. Ou was supported by a grant from the Simons Foundation ($\#427231$, Ye-Lin Ou).\\

2. Conflict of interest: The authors declare that they have no conflict of interest.

\end{document}